\documentclass[11pt,a4paper]{article}
\usepackage{amsmath}
\usepackage{amsfonts}
\usepackage{amscd}
\usepackage{amsthm}
\usepackage{enumerate}
\newtheorem{Thm}{Theorem}[section]
\newtheorem{Lem}[Thm]{Lemma}
\newtheorem{Prop}[Thm]{Proposition}
\newtheorem{Cor}[Thm]{Corollary}
\newtheorem{Conj}[Thm]{Conjecture}
\theoremstyle{definition}
\newtheorem{Def}[Thm]{Definition}
\newtheorem{Rem}[Thm]{Remark}
\newcommand{\Sw}{\mathop{\mathrm{Sw}}\nolimits}

\newcommand{\Ind}{\mathop{\mathrm{Ind}}\nolimits}
\newcommand{\Trbr}{\mathrm{Tr}^{\mathrm{Br}}}
\newcommand{\rk}{\mathop{\mathrm{rk}}\nolimits}
\newcommand{\Spec}{\mathop{\mathrm{Spec}}\nolimits}
\newcommand{\Tr}{\mathop{\mathrm{Tr}}\nolimits}

\newcommand{\id}{\mathrm{id}}
\newcommand{\Gal}{\mathrm{Gal}}
\title{Wild Ramification and Restrictions to Curves}
\author{Hiroki Kato}
\begin{document}
\maketitle
\begin{abstract}
We prove that wild ramification of a constructible sheaf on a surface 
is determined by that of the restrictions to all curves. 
We deduce from this result 
that the Euler-Poincar\'e characteristic of a constructible sheaf 
on a variety of arbitrary dimension over an algebraically closed field 
is determined by wild ramification of the restrictions to all curves. 
We similarly deduce from it that 
so is the alternating sum of the Swan conductors of the cohomology groups, 
for a constructible sheaf on a variety over a local field. 
\end{abstract}
The Euler-Poincar\'e characteristic of a constructible \'etale sheaf is determined by wild ramification of the sheaf \cite{I} 
	and so is the alternating sum of the Swan conductors of the cohomology groups \cite{V}. 
Deligne-Illusie formulated the notion ``same wild ramification'' for constructible sheaves on a variety over a field 
	using the Brauer trace 
	and proved that constructible sheaves have the same Euler-Poincar\'e characteristics
	if they have the same wild ramification 
	\cite[Th\'eor\`eme 2.1]{I}. 
Vidal proved that 
	for constructible sheaves on a variety over a local field, 
	if they have the same wild ramification, 
	then the alternating sums of the Swan conductors of the cohomology groups 
		is the same \cite[Corollaire 3.4]{V}. 


For the notion ``same wild ramification'', 
	Saito-Yatagawa gave a formulation which is weaker than that of Deligne-Illusie 
using, instead of the Brauer trace, the dimensions of fixed parts \cite[Definition 5.1]{SY}. 
Having the same wild ramification in their sense 
	also implies having the same Euler-Poincar\'e characteristics \cite[Proposition 0.2]{SY} 
and they proved that 
constructible sheaves on a smooth variety have the same characteristic cycles 
if the sheaves have the same wild ramification \cite[Theorem 0.1]{SY}. 

Inspired by Beilinson's suggestion that the characteristic cycle of a constructible sheaf 
be determined by wild ramification of the restrictions to all curves, 
we consider whether wild ramification of a constructible sheaf is determined by that of the restrictions to all curves: 
\begin{Conj}
\label{conj}
Let $S$ be an excellent noetherian scheme whose closed points have perfect residue fields, 
$X$ an $S$-scheme separated of finite type, 
$\Lambda$ and $\Lambda'$ finite fields of characteristics invertible on $S$, 
and $\mathcal{F}$ and $\mathcal{F}'$ constructible complexes of 
$\Lambda$-modules and $\Lambda'$-modules respectively on $X$. 
Then the followings are equivalent. 
\begin{enumerate}
\renewcommand{\labelenumi}{{\rm(\roman{enumi})}}
\item $\mathcal{F}$ and $\mathcal{F}'$ have the same virtual wild ramification over $S$. 
\item $\mathcal{F}$ and $\mathcal{F}'$ have universally the same conductors over $S$. 
\end{enumerate}
\end{Conj}


The terminology {\it same virtual wild ramification} is defined in Definition \ref{svwr} in the text, 
	which is a slight modification of \cite[Definition 5.1]{SY}. 
The terminology {\it universally the same conductors}, 
which means wild ramification of the restrictions to all curves are the same, 
is defined in Definition \ref{sc} in the text. 
The implication (i)$\Rightarrow$(ii) is straightforward (Proposition \ref{W to A}). 

We prove that Conjecture \ref{conj} holds if $X$ is an open subscheme of a proper $S$-scheme of dimension $\leq2$, 
that is, the wild ramification of a constructible sheaf on a surface is determined by that of the restrictions to all curves (Theorem \ref{A to W}). 
The hypothesis that dimension $\leq2$ is used to take a regular compactification, 
and if we assume resolution of singularities is always possible, 
Conjecture \ref{conj} holds in general (Proposition \ref{p-power case}). 

From the above result we deduce the followings: 
\begin{Thm}
\label{A to X}
Let 
$X$ be a separated scheme of finite type over an algebraically closed field 
of characteristic $p>0$. 
Let $\Lambda$ and $\Lambda'$ be finite fields of characteristics different from $p$ 
and $\mathcal{F}$ and $\mathcal{F}'$ constructible complexes 
of $\Lambda$-modules and $\Lambda '$-modules respectively on $X$. 
Then, if $\mathcal{F}$ and $\mathcal{F}'$ 
have universally the same conductors, 
then we have 
$\chi_c(X,\mathcal{F})=\chi_c(X,\mathcal{F}')$.
\end{Thm}
\begin{Thm}
\label{sc to sgc}
Let $S$ be an excellent trait with perfect residue field of characteristic $p>0$ 
and with generic point $\eta=\Spec K$ 
and $X$ a separated scheme of finite type over $K$. 
Let $\Lambda$ and $\Lambda'$ be finite fields of characteristics different from $p$ 
and $\mathcal{F}$ and $\mathcal{F}'$ constructible complexes 
of $\Lambda$-modules and $\Lambda'$-modules respectively on $X$. 
If $\mathcal{F}$ and $\mathcal{F}'$ have universally the same conductors over $S$, 
then we have $\Sw(X,\mathcal{F})=\Sw(X,\mathcal{F}')$, 
where $\Sw(X,-)$ denotes the alternating sum 
$\sum_{i}(-1)^i\Sw H_c^i(X_{\overline{K}},-)$ of the Swan conductors. 
\end{Thm}

The proof of Theorem \ref{A to X} is reduced to the case where $X$ is a surface.
Similarly that of Theorem \ref{sc to sgc} is reduced to the case where $X$ is a curve over $K$. 
Then, we can apply Theorem \ref{A to W}. 
Therefore, Theorem \ref{A to X} follows from the fact that 
having the same virtual wild ramification implies having the same Euler-Poincar\'e characteristics (Proposition \ref{W to X})
and Theorem \ref{sc to sgc} follows from the fact that 
having the same virtual wild ramification implies having ``the same Swan conductors" (Proposition \ref{svwr to sgc}). 


\paragraph{Acknowledgment}
The author thanks Professor Takeshi Saito 
for leading him to the study of arithmetic geometry, 
for a lot of discussion and a lot of advice, 
and for sharing the suggestion of Beilinson. 
He also thanks Yuri Yatagawa for answering many questions on ramification theory. 

\paragraph{Notations and terminologies}
We fix some notations and terminologies; 
Throughout this paper, we fix finite fields $\Lambda$ and $\Lambda'$ 
and assume that the characteristics of $\Lambda$ and $\Lambda'$ are invertible on all schemes considered in this paper.

For an excellent noetherian scheme $S$, an $S$-{\it curve} (resp. an $S$-{\it surface}) means 
an open subscheme of a proper $S$-scheme of dimension 1 (resp. 2).  

A {\it smooth} sheaf means a locally constant constructible (\'etale) sheaf.

A {\it constructible} complex (resp. {\it smooth} complex) means 
a complex $\mathcal{F}$ of sheaves whose cohomology sheaves $\mathcal{H}^q(\mathcal{F})$ are zero except for finitely many $q$ and constructible (resp. smooth) for all $q$. 


For a scheme $X$ and a geometric point $x\to X$ we denote the strict localization at $x$ by $X_{(x)}$. 

For a pro-finite group $G$, 
we denote by $K(\Lambda[G])$ the Grothendieck group of the category 
of finite $\Lambda$-vector spaces on which $G$ acts continuously 
with respect to the discrete topology of $\Lambda$. 

For a smooth complex $\mathcal{F}$ on a connected noetherian scheme $X$ 
and for a $G$-torsor $W\to X$, 
we say $\mathcal{F}$ is {\it trivialized} by $W\to X$ 
if every cohomology sheaf $\mathcal{H}^q(\mathcal{F}|_W)$ is constant. 
If $\mathcal{F}$ is {\it trivialized} by a $G$-torsor $W\to X$, 
$\mathcal{F}$ defines a virtual representation $M$ of $G$. 
We call this $M$ the virtual representation of $G$ {\it corresponding} to $\mathcal{F}$. 
If we choose a geometric point $y$ of $X$, 
the stalk $\mathcal{F}_y$ defines an element of $K(\Lambda[\pi_1(X,y)])$, 
which we denote by $[\mathcal{F}_y]$.


\section{Preliminaries on the Swan conductor}
\label{swbr}
In this section, we see some elementary properties of the Swan conductor. 

We need the following lemma on the Brauer trace. 
For the definition of the Brauer trace we refer to \cite[18.1]{Se} 
and denote it by $\Trbr$. 
Let $p$ be a prime number 
and $\Lambda$ a finite field of characteristic different from $p$. 
\begin{Lem}[Lemma 4.1 in \cite{SY}]
\label{brtr}
Let 
$M$ be a $\Lambda$-vector space of finite dimension 
and $\sigma$ an endomorphism of $M$ of order a power of $p$. 
Then for a subfield $K$ of the fraction field of the Witt ring $W(\Lambda)$ 
which is finite degree over $\mathbb{Q}$ 
and contains $\Trbr(\sigma,M)$, 
we have 
$$
\frac{1}{[K:\mathbb{Q}]}\Tr_{K/\mathbb{Q}}\Trbr(\sigma,M)=\frac{1}{p-1}(p\cdot\dim M^{\sigma}-\dim M^{\sigma^p}).
$$
\end{Lem}

Let $G$ be a pro-finite group. 
For a pro-$p$ subgroup $H$ of $G$, 
we define the homomorphism $\Gamma^H:K(\Lambda[G])\to\mathbb{Z}$ 
by assigning the class of each finitely generated $\Lambda[G]$-module $M$ to $\dim M^H$, 
where $M^H$ denotes the fixed part. 

Let $K$ be a henselian discrete valuation field with algebraically closed residue field of characteristic $p$ 
and $L$ be a finite Galois extension of $K$ with Galois group $G$. 
The Swan character $sw_G:G\to\mathbb{Z}$ is defined in \cite[19.1]{Se}. 
Let $M$ be an element of $K(\Lambda[G])$. 
The Swan conductor $\Sw(M)$ of $M$ is defined in \cite[19.3]{Se}. 

\begin{Lem}
\label{sw}
We have 
$$
\Sw(M)=\frac{1}{|G|}\sum_{\sigma\in P}sw_G(\sigma)\cdot\frac{p\cdot\Gamma^\sigma(M)-\Gamma^{\sigma^p}(M)}{p-1},
$$
where $P$ is the $p$-Sylow subgroup of $G$. 
\end{Lem}

\begin{proof}
We may assume $M$ is a $\Lambda$-vector space with action of $G$. We have 
$$
\Sw(M)=\frac{1}{|G|}\sum_{\sigma\in P}sw_G(\sigma)\cdot\Trbr(\sigma,M).
$$
\cite[19.3]{Se}.
Since $\Sw(M)$ is an integer, the assertion follows from Lemma \ref{brtr}. 
\end{proof}

Let $K'$ be an intermediate extension of $L/K$ with Galois group $H=\Gal(L/K')$ 
and $M\in K(\Lambda[G])$ a virtual representation of $G$. 
We regard the restriction $M|_H$ as a representation of the absolute Galois group of $K'$, 
and hence $\Sw(M|_H)$ denotes the Swan conductor with respect to $K'$. 

\begin{Lem}
\label{if the same sw then the same fixed parts}
Assume that the extension $L/K$ is cyclic of degree $p^e$ for an integer $e\geq0$. 
For virtual representations $M\in K(\Lambda[G])$ and $M'\in K(\Lambda'[G])$, 
if $\rk(M)=\rk(M')$ and if $\Sw(M|_H)=\Sw(M'|_H)$ for every subgroup $H$ of $G$, 
then we have $\Gamma^\sigma(M)=\Gamma^\sigma(M')$ for every $\sigma\in G$. 
\end{Lem}

\begin{proof}
We prove the assertion by induction on $e$. 
The $e=0$ case is clear. 
Suppose $e\geq1$. 
Consider the subgroup $G^p$ of $G$ of index $p$. 
Then we can apply the induction hypothesis to $M|_{G^p}$ and $M'|_{G^p}$, 
and hence $\Gamma^g(M)=\Gamma^g(M')$ for every $g\in G^p$. 
Let $\sigma$ be a generator of the cyclic subgroup $G$. 
By Lemma \ref{sw}, 
$$
\Sw(M)=\frac{1}{|G|}\sum_{g\in G}sw_G(g)\cdot\frac{p\cdot\Gamma^g(M)-\Gamma^{g^p}(M)}{p-1},
$$
and similarly for $\Sw(M')$. 
Here, by the assumption, $\Sw(M)=\Sw(M')$. 
Since $\Gamma^g(M)=\Gamma^g(M')$ for every $g\in G^p$, 
we have 
$$
\left(\sum_{g\in G\setminus G^p}sw_G(g)\right)\cdot\Gamma^\sigma(M)=
\left(\sum_{g\in G\setminus G^p}sw_G(g)\right)\cdot\Gamma^\sigma(M).
$$
The extension $L/K$ is of degree a power of $p$ and hence $sw_G(g)>0$ for every $g\in G\setminus\{1\}$. 
Thus we have $\Gamma^\sigma(M)=\Gamma^\sigma(M')$ and the assertion follows. 
\end{proof}

We denote the absolute Galois group of a field $F$ by $G_F$. 
In Section 5, we will use the induction formula for the Swan conductor: 
\begin{Lem}
\label{conductor discriminant}
Let $K'$ be a finite extension of $K$ 
and $M'$ a continuous representation of $G_{K'}$ with respect to the discrete topology of $M'$. 
Then we have 
$$\Sw (\Ind_{G_{K'}}^{G_K}M')=\Sw M'+\dim M'\cdot(d+[K'_s:K]-1),$$ 
where 
$\Ind_{G_{K'}}^{G_K}M'$ denotes the induced representation, 
$K'_s$ the maximal separable extension in $K'/K$, 
and $d$ the length of the discriminant of the extension $K'_s/K$. 

\end{Lem}
\begin{proof}
If $K'/K$ is separable then the assertion follows from  \cite[Proposition 1.(c)]{R}. 
Hence we may assume $K'/K$ is purely inseparable. 
Here, $\Ind_{G_{K'}}^{G_K}M'$ of $G_K$ is 
the representation induced by $M'$ and by the isomorphism $G_{K'}\to G_K$. 
Hence, by Lemma \ref{p insep} below, we have $sw_{L/K}=sw_{L'/K'}$ for finite Galois extension $L$ of $K$ 
and for the composite $L'$ of $L$ and $K'$. Thus the assertion follows. 
\end{proof}

\begin{Lem}
\label{p insep}
Let $K'$ be a finite purely inseparable extension of $K$ of degree $p^s$. 
Let $L$ be a finite Galois extension of $K$ with Galois group $G$ 
and $L'$ a composite of $L$ and $K'$. 
Take a uniformizer $\pi$ of $L$. 
Then $\pi^{1/p^s}$ belongs to $L'$ and is a uniformizer of $L'$ 
and we have $v(\sigma\pi-\pi)=v'(\sigma\pi^{1/p^s}-\pi^{1/p^s})$ 
for any element $\sigma\neq1$ of $G$, 
where $v$ and $v'$ are the valuation of $L$ and $L'$ respectively. 
\end{Lem}
\begin{proof}
We may assume $K$ is complete with respect to its valuation. 
Since every complete discrete valuation ring is excellent we can apply Lemma \ref{p insep ext of cdvf} below 
and hence $\pi^{1/p^s}$ belongs to $L'$ and is a uniformizer of $L'$. 
The $p^s$-th power map $L'\to L$ induces an isomorphism 
	between their rings of integers 
	and hence we have $v(x^{p^s})=v'(x)$ for every $x\in L'$. 
Thus the assertion follows. 
\end{proof}
\begin{Lem}
\label{p insep ext of cdvf}
Let $K$ be a discrete valuation field of characteristic $p>0$ with perfect residue field 
and $\pi$ a uniformizer of $K$. 
Assume that the ring of integer $\mathcal{O}_K$ is excellent. 
Then any finite purely inseparable extension of $K$ 
is of the form $K(\pi^{1/p^s})$ for some integer $s\geq0$. 
\end{Lem}
\begin{proof}
We note that $K^p$ is the field of fraction of the excellent discrete valuation ring $\mathcal{O}_K^p$ 
	and that $\mathcal{O}_K$ is the integral closure of $\mathcal{O}_K^p$ in $K$, 
	which is finite over $\mathcal{O}_K^p$. 
Since the ramification index of the extension $K/K^p$ is $p$ 
and the extension of the residue fields is trivial, 
we have $[K:K^p]=p$ and hence $[K^{1/p^s}:K]=p^s$ for every $s\geq0$. 
Since every finite purely inseparable extension of $K$ of degree $p^s$ 
	is contained in $K^{1/p^s}$, the assertion follows. 
\end{proof}



\section{Same wild ramification and Same conductors}
Let $S$ be an excellent noetherian scheme. 
We always assume that 
closed points of the base scheme $S$ have perfect residue fields. 
We define the terminologies {\it same virtual wild ramification} and {\it same conductors} 
for complexes on a scheme over $S$. 
The terminology {\it same virtual wild ramification} 
is a modification of the terminology {\it same wild ramification}  
in \cite[Definition 5.1]{SY}. 

\subsection{Same wild ramification}
In this subsection, we do not use the assumption 
that closed points of $S$ have perfect residue fields. 

We note that by Nagata's compactification theorem, 
separated scheme of finite type over a noetherian scheme have a compactification. 

We say an element $\sigma$ of a pro-finite group $G$ is {\it pro-p} 
if the closed subgroup $\overline{\langle\sigma\rangle}$ of $G$ is pro-$p$. 
For a pro-$p$ element $\sigma\in G$, we put $\Gamma^\sigma=\Gamma^{\overline{\langle\sigma\rangle}}$. 
For the definition of $\Gamma^{\overline{\langle\sigma\rangle}}$, see Section \ref{swbr}. 

For the spectrum $x=\Spec F$ of a field $F$, 
we denote the characteristic of $F$ by $p_x$. 

\begin{Def}
Let $\overline{X}$ be a noetherian normal connected scheme 
and $X$ a dense open subscheme of $\overline{X}$. 
For smooth complexes $\mathcal{F}$ and $\mathcal{F}'$ 
of $\Lambda$-modules and $\Lambda'$-modules respectively on $X$, 
we say $\mathcal{F}$ and $\mathcal{F}'$ have {\it the same virtual wild ramification} along $\overline{X}\setminus X$ 
if for any geometric point $x\to \overline{X}$ 
and any pro-$p_x$ element $\sigma$ 
of $\pi_1(\overline{X}_{(x)}\times_{\overline{X}}X,y)$, 
where $y$ is a geometric point, 
we have $\Gamma ^{\sigma}([\mathcal{F}_y])=\Gamma ^{\sigma}([\mathcal{F}'_y])$. 
\end{Def}

To show $\mathcal{F}$ and $\mathcal{F}'$ have the same virtual wild ramification along $\overline{X}\setminus X$, 
it suffices to check the above condition for a geometric point over every closed point, 
i.e. it suffices to check that 
for a geometric point $x\to \overline{X}$ {\it over every closed point} 
and every pro-$p_x$ element $\sigma$ 
of $\pi_1(\overline{X}_{(x)}\times_{\overline{X}}X,y)$, 
where $y$ is a geometric point, 
we have $\Gamma ^{\sigma}([\mathcal{F}_y])=\Gamma ^{\sigma}([\mathcal{F}'_y])$. 
Note that 
for a geometric point $x'$ of $\overline{X}$ and a geometric point $x$ which specializes $x'$, 
if we choose a geometric point $y'$ of $\overline{X}_{(x')}\times_{\overline{X}}X$, 
the specialization morphism $\overline{X}_{(x')}\to\overline{X}_{(x)}$ 
induces a homomorphism 
$\pi_1(\overline{X}_{(x')}\times_{\overline{X}}X,y')\to\pi_1(\overline{X}_{(x)}\times_{\overline{X}}X,y)$, 
where $y$ is the image of $y'$. 

\begin{Def}[cf. Definition 5.1 in \cite{SY}]
\label{svwr}
Let $X$ be an $S$-scheme separated of finite type 
and $\mathcal{F}$ and $\mathcal{F}'$ constructible complexes 
of $\Lambda$-modules and $\Lambda '$-modules respectively on $X$. 
\begin{enumerate}[(i)]
\item
\label{sm case}
Assume that $X$ is normal and connected 
and that $\mathcal{F}$ and $\mathcal{F}'$ are smooth. 
We say $\mathcal{F}$ and $\mathcal{F}'$ have the {\it same virtual wild ramification} over $S$ 
if there exist 
a normal compactification $\overline{X}$ of $X$ over $S$ 
such that $\mathcal{F}$ and $\mathcal{F}'$ have the same virtual wild ramification along $\overline{X}\setminus X$. 
\item 
We say $\mathcal{F}$ and $\mathcal{F}'$ have the {\it same virtual wild ramification} over $S$ 
if there exists a decomposition $X=\coprod_i X_i$ 
such that 
$X_i$ is locally closed connected normal subscheme 
and that $\mathcal{F}|_{X_i}$ and $\mathcal{F}'|_{X_i}$ are smooth complexes  
and have the same virtual wild ramification in the sense of (\ref{sm case}) for every $i$. 
\end{enumerate}
\end{Def}
If a $G$-torsor is given, 
we use the terminology same virtual wild ramification 
also for virtual representations of $G$.

\begin{Rem}
	\begin{enumerate}[(i)]
\item
Our definition of ``same virtual wild ramification'' is weaker than that in \cite{I} and \cite{V}. 
Consider the case where $S$ is a spectrum of a field of characteristic $p>0$ 
or that of a discrete valuation ring with residue field of characteristic $p>0$. 
In \cite[D\'efinition 2.3.1]{V}, 
for smooth complexes $\mathcal{F}$ and $\mathcal{F}'$ of $\Lambda$-modules 
on a normal connected scheme $X$ separated of finite type over $S$, 
the property that $\mathcal{F}$ and $\mathcal{F}'$ have the same virtual wild ramification 
is defined by the following condition: 
There exists a normal compactification $\overline{X}$ of $X$ over $S$ 
such that for every geometric point $x$ of $\overline{X}$ 
and every pro-$p$ element $\sigma\in\pi_1(\overline{X}_{(x)}\times_{\overline{X}}X,y)$, 
we have 
$\Trbr(\sigma,[\mathcal{F}_y])=\Trbr(\sigma,[\mathcal{F}'_y])$. 
Thus our ``same virtual wild ramification'' is weaker than that in \cite{V}. 
Indeed, for a finite group $G$, 
an element $\sigma\in G$ of order a power of $p$, 
and an element $M\in K(\Lambda[G])$, 
we have 
$$\Gamma^\sigma(M)=\frac{1}{|\langle\sigma\rangle|}\sum_{g\in\langle\sigma\rangle}\Trbr(g,M)$$ 
(c.f. \cite[1.4.7]{I}).

\item
If $S$ is a spectrum of a field and 
$\mathcal{F}$ and $\mathcal{F}'$ are constructible sheaves, 
our ``same virtual wild ramification'' coincides with 
``same wild ramification'' in \cite[Definition 5.1]{SY}. 
\end{enumerate}
\end{Rem}

\begin{Lem}
\label{pullback}
Let $S'\to S$ be a morphism of excellent noetherian schemes. 
Consider a commutative diagram
$$
\begin{CD}
X'@>{g}>>X\\
@VVV@VVV\\
S'@>>>S
\end{CD}
$$
such that the vertical arrows are separated of finite type. 
Let $\mathcal{F}$ and $\mathcal{F}'$ be constructible complexes of sheaves of $\Lambda$-modules and $\Lambda'$-modules respectively on $X$. 
If $\mathcal{F}$ and $\mathcal{F}'$ have the same virtual wild ramification over $S$, 
then $g^*\mathcal{F}$ and $g^*\mathcal{F}'$ have the same virtual wild ramification over $S'$. 
\end{Lem}

\begin{proof}
We note that the normal locus of an excellent scheme is open. 
By devissage, we may assume 
$X$ and $X'$ are normal and connected  
and $\mathcal{F}$ and $\mathcal{F}'$ are smooth complexes. 
Then we take a normal compactification $\overline{X}$ of $X$ over $S$ 
such that $\mathcal{F}$ and $\mathcal{F}'$ have the same virtual wild ramification along $\overline{X}\setminus X$. 
We can take a normal compactification $\overline{X'}$ of $X'$ over $S'$ 
with a morphism $\bar{g}:\overline{X'}\to\overline{X}$ extending $g$. 
Then, 
for a geometric point $x'\to\overline{X'}$, 
a geometric point $y'$ of $\overline{X'}_{(x')}\times_{\overline{X'}}X'$, 
and a pro-$p_{x'}$ element $\sigma'$ of $\pi_1(\overline{X'}_{(x')}\times_{\overline{X'}}X',y')$, 
we have $\Gamma^{\sigma'}([(g^*\mathcal{F})_{y'}])=\Gamma^{\bar{g}_*(\sigma')}([\mathcal{F}_{\bar{g}(y')}])$ 
and the same equality for $\mathcal{F}'$, 
where $\bar{g}_*$ is the induced homomorphism 
$\pi_1(\overline{X'}_{(x')}\times_{\overline{X'}}X',y')\to\pi_1(\overline{X}_{(\bar{g}(x'))}\times_{\overline{X}}X,\bar{g}(y'))$. 
Here since $\bar{g}_*(\sigma')$ is a pro-$p_{\bar{g}(x')}$ element of 
$\pi_1(\overline{X}_{(\bar{g}(x'))}\times_{\overline{X}}X,\bar{g}(y'))$, 
we have $\Gamma^{\sigma'}([\mathcal{F}_{y'}])=\Gamma^{\sigma'}([\mathcal{F}'_{y'}])$. 
\end{proof}

\subsection{Same conductors}
For a trait $T$ 
with generic geometric point $\eta$ and closed point $t$ such that the residue field of $t$ is algebraically closed  
and for a constructible complex $\mathcal{F}$ of $\Lambda$-modules on $T$, 
the Artin conductor $a(\mathcal{F})$ is defined by 
$a(\mathcal{F})=\rk(\mathcal{F}_{\eta})-\rk(\mathcal{F}_{t})+\Sw(\mathcal{F}_{\eta})$.  

For a regular scheme $X$ of dimension 1 whose closed points have perfect residue fields, 
a constructible complex $\mathcal{F}$ of $\Lambda$-modules on $X$, 
and a geometric point $x$ over a closed point of $X$, 
the Artin conductor $a_x(\mathcal{F})$ at $x$ is defined by $a_x(\mathcal{F})=a(\mathcal{F}|_{X_{(x)}})$. 

For a geometric point $x$ over a closed point of $X$, 
we denote the generic point of $X_{(x)}$ by $\eta_x$ 
and a geometric point over $\eta_x$ by $\bar{\eta_x}$. 
For a constructible complex $\mathcal{G}$ on a dense open subscheme $U$ of $X$ 
the Swan conductor $\Sw_x(\mathcal{G})$ at $x$ is defined to be 
the Swan conductor of the virtual representation $[\mathcal{G}_{\bar{\eta_x}}]$ 
of the absolute Galois group $\Gal(\bar{\eta_x}/\eta_x)$. 

Let $S$ be an excellent noetherian scheme. 
By Zariski's main theorem, for a regular $S$-curve $C$, 
there exists a dense open immersion $j:C\to\overline{C}$ over $S$ 
to a regular scheme $\overline{C}$ proper over $S$. 
We call such an open immersion a canonical regular compactification of $C$ over $S$. 
It is unique up to unique isomorphism. 

We assume that every closed point of $S$ has perfect residue field. 
\begin{Def}
\label{sc}
Let 
$X$ be an $S$-scheme separated of finite type 
and $\mathcal{F}$ and $\mathcal{F}'$ constructible complexes of $\Lambda$-modules and $\Lambda'$-modules respectively on $X$. 
We say $\mathcal{F}$ and $\mathcal{F}'$ have {\it universally the same conductors} over $S$ 
if for every regular $S$-curve $C$, 
every $S$-morphism $g:C\to X$, 
every canonical regular compactification $j:C\to\overline{C}$ of $C$ over $S$
and every geometric point $v\to \overline{C}$ over a closed point, 
we have $a_v(j_!g^*\mathcal{F})=a_v(j_!g^*\mathcal{F}')$. 
\end{Def}

\begin{Lem}
\label{char of sc}
$\mathcal{F}$ and $\mathcal{F}'$ have universally the same conductors over $S$ 
if and only if the following two conditions hold; 
\begin{enumerate}[{\rm (i)}]
\item
\label{rk}
for every geometric point $x$ of $X$, 
we have $\rk(\mathcal{F}_x)=\rk(\mathcal{F}'_x)$, 

\item
\label{swc}
for every regular $S$-curve $C$, 
every $S$-morphism $g:C\to X$, 
every canonical compactification $C\to\overline{C}$, 
and every geometric point $v\to \overline{C}$ over a closed point, 
we have $\Sw_v(g^*\mathcal{F})=\Sw_v(g^*\mathcal{F}')$. 
\end{enumerate}
\end{Lem}

\begin{proof}
The sufficiency is clear. 
We prove the necessity. 
Assume $\mathcal{F}$ and $\mathcal{F}'$ have universally the same conductors over $S$. 
Let $x$ be a geometric point of $X$. 
Take a stratification $X=\coprod_i X_i$ 
	such that $\mathcal{F}|_{X_i}$ and $\mathcal{F}'|_{X_i}$ are smooth for every $i$ 
	and a covering $W\to X_j$ trivializing $\mathcal{F}|_{X_j}$ and $\mathcal{F}'|_{X_j}$, 
		where $X_j$ is the subset over which $x$ lies. 
Then, for every regular $S$-curve $C$ which is not proper and an $S$-morphism $g:C\to W$ 
and every geometric point $v$ of $\overline{C}\setminus C$, 
where $j:C\to\overline{C}$ is a canonical compactification over $S$, 
we have 
$a_v(j_!g^*(\mathcal{F}|_W))=\rk(\mathcal{F}_x)$, and similarly for $\mathcal{F}'$. 
Thus, by the assumption we get $\rk(\mathcal{F}_x)=\rk(\mathcal{F}'_x)$. 
By the definition of the Artin conductor 
	we have $\Sw_v(g^*\mathcal{F})=a_v(j_!g^*\mathcal{F})-\rk(\mathcal{F}_{g(v)})$ 
	and (ii) also follows. 
\end{proof}



\section{The main theorem}
Let $S$ be an excellent noetherian scheme and 
assume that all closed points of $S$ have perfect residue fields. 
\subsection{Statement of the main theorem}
It is easy to show that 
having the same virtual wild ramification implies having universally the same conductors: 
\begin{Prop}
\label{W to A}
Let $X$ be a scheme separated of finite type over $S$ 
and $\mathcal{F}$ and $\mathcal{F}'$ constructible complexes of 
$\Lambda$-modules and $\Lambda'$-modules respectively on $X$. 
Assume $\mathcal{F}$ and $\mathcal{F}'$ have the same virtual wild ramification over $S$. 
Then they have universally the same conductors over $S$. 
\end{Prop}
\begin{proof}
By Lemma \ref{pullback} and Lemma \ref{char of sc}, 
we may assume that 
$X$ is a regular connected $S$-curve and $\mathcal{F}$ and $\mathcal{F}'$ are smooth 
and it suffices to show $\Sw_x(\mathcal{F})=\Sw_x(\mathcal{F}')$ 
for every geometric point $x$ of $\overline{X}\setminus X$, 
where $X\to\overline{X}$ is a canonical regular compactification of $X$ over $S$. 
We may further assume $S=\overline{X}$. 
Again by Proposition \ref{pullback}, we may assume 
$S$ is a strictly local trait and $X$ is the generic point of $S$. 
Then the assertion follows from Lemma \ref{sw}. 
\end{proof}

We now state our main theorem; 
Conjecture \ref{conj} holds for sheaves on an $S$-surface: 
\begin{Thm}
\label{A to W}
Let $X$ be an $S$-surface 
and $\mathcal{F}$ and $\mathcal{F}'$ constructible complexes of 
$\Lambda$-modules and $\Lambda'$-modules respectively on $X$. 
Then the followings are equivalent. 
\begin{enumerate}[{\rm (i)}]
\item
\label{thmswr}
$\mathcal{F}$ and $\mathcal{F}'$ have the same virtual wild ramification over $S$. 
\item
\label{thmsc}
$\mathcal{F}$ and $\mathcal{F}'$ have universally the same conductors over $S$. 
\end{enumerate}
\end{Thm}
We show that, 
under the existence of a regular compactification, 
Conjecture \ref{conj} holds for smooth complexes on a regular $S$-scheme: 
\begin{Prop}
\label{p-power case}
Let $X$ be a regular connected scheme separated of finite type over $S$ 
and $\mathcal{F}$ and $\mathcal{F}'$ smooth complexes of 
$\Lambda$-modules and $\Lambda'$-modules respectively on $X$. 
Take a finite group $G$ and a $G$-torsor $W\to X$ 
which trivializes $\mathcal{F}$ and $\mathcal{F}'$. 
We assume that $\mathcal{F}$ and $\mathcal{F}'$ have universally the same conductors over $S$ 
and that 
for every subgroup $H$ of $G$ which is maximal among cyclic subgroups	of prime-power order, 
the quotient $W/H$ has a regular compactification over $S$. 
Then 
$\mathcal{F}$ and $\mathcal{F}'$ have the same virtual wild ramification over $S$. 
\end{Prop}

\subsection{Curve case}
First, we show Proposition \ref{p-power case} holds if $\dim\overline{X}\leq1$ and if $G$ is of prime-power order: 
\begin{Lem}
\label{curve case}
Let $X$ be a regular $S$-curve 
and $\mathcal{F}$ and $\mathcal{F}'$ smooth complexes 
of $\Lambda$-modules and $\Lambda'$-modules respectively on $X$. 
Assume that 
$\mathcal{F}$ and $\mathcal{F}'$ 
are trivialized by a $\mathbb{Z}/p^e\mathbb{Z}$-torsor over $X$ 
for a prime number $p$ and an integer $e\geq0$ 
and have universally the same conductors over $S$. 
Then they have the same virtual wild ramification over $S$. 
\end{Lem}

\begin{proof}
Let $W\to X$ be a $G=\mathbb{Z}/p^e\mathbb{Z}$-torsor trivializing $\mathcal{F}$ and $\mathcal{F}'$ 
and $X\to\overline{X}$ a canonical regular compactification of $X$ over $S$. 

Take a geometric point $x$ of $\overline{X}$. 
We show that for every pro-$p_x$ element 
$\sigma\in\pi_1(\overline{X}_{(x)}\times_{\overline{X}}X,y)$, 
we have $\Gamma^\sigma([\mathcal{F}_y])=\Gamma^\sigma([\mathcal{F}'_y])$, 
where $y$ is a geometric point. 
If $x$ lies over a point of $X$, this is trivial. 
We assume $x$ lies over a point of $\overline{X}\setminus X$. 

We note that 
$K=\Gamma(\overline{X}_{(x)}\times_{\overline{X}}X,\mathcal{O})$ 
is a strictly henselian discrete valuation field 
and that 
$\overline{W}_{(w)}\times_{\overline{W}}W\to\overline{X}_{(x)}\times_{\overline{X}}X$ 
is an $I_x$-torsor, 
where $\overline{W}$ is the normalization of $\overline{X}$ in $W$, 
$w$ is a geometric point of $\overline{W}$ which lies over $x$, 
and $I_x$ is the inertia group, i.e. the stabilizer of $w$. 
The elements $[\mathcal{F}_y]$ and $[\mathcal{F}'_y]$ 
of $K(\Lambda[\pi_1(\overline{X}_{(x)}\times_{\overline{X}}X,y)])$ come from 
elements $M$ and $M'$ of $K(\Lambda[I_x])$. 

By the assumption that 
$\mathcal{F}$ and $\mathcal{F}'$ have universally the same conductors, 
it follows that for every subgroup $H$ of $I_x$ we have 
$\Sw(M|_H)=\Sw(M'|_H)$. 
In fact, if we put $V=W/H$, 
then $\overline{W}_{(w)}\times_{\overline{W}}W\to\overline{V}_{(v)}\times_{\overline{V}}V$ 
is an $H$-torsor, 
where $\overline{V}$ is the normalization of $\overline{X}$ in $V$ 
and $v$ is the geometric point under $w$. 
Then the assertion follows from Lemma \ref{if the same sw then the same fixed parts}.
\end{proof}
\subsection{Existence of good curve}
The following lemma will be used to reduce the proof of Proposition \ref{p-power case} to the curve case.
Let $p$ be a prime number. 
\begin{Lem}
\label{taking a good curve}
Let 
$X$ be a dense open subscheme of 
a regular connected excellent noetherian scheme $\overline{X}$. 
Let $G$ be a cyclic group of $p$-power order 
and $W\to X$ a $G$-torsor. 
Then, for any geometric point $x$ over a closed point of $\overline{X}$, 
there exists 
	a regular scheme $\overline{C}$ of dimension one, 
	a finite morphism $\bar{g}:\overline{C}\to \overline{X}$, 
	and a geometric point $v$ of $\overline{C}$ over $x$ 
such that 
$C=X\times_{\overline{X}}\overline{C}$ is not empty 
and the composite 
$\pi_1(\overline{C}_{(v)}\times_{\overline{C}}C,y')
\to\pi_1(\overline{X}_{(x)}\times_{\overline{X}}X,y)\to I_x$ 
is surjective, 
where $I_x$ is an inertia subgroup of $G$ at $x$ 
and where $y'$ is a geometric point $\overline{C}_{(v)}\times_{\overline{C}}C$ 
and $y$ is its image. 
\end{Lem}

To prove this we use the following lemma: 
\begin{Lem}[Lemma 2.4 in \cite{KS}]
\label{KS}
Let $X=\Spec A$ be a spectrum of a normal excellent noetherian local ring $A$. 
Let $U$ be a dense open subscheme of $X$ 
and $V\to U$ an $\mathbb{F}_p$-torsor ramified along a closed subscheme $D\subset X$ of codimension one. 
Then, there exists a closed subscheme $C\subset X$ of dimension one 
such that the $\mathbb{F}_p$-torsor $V\times_X\widetilde{C}\to U\times_X\widetilde{C}$ is ramified along $D\times_X\widetilde{C}$, 
where $\widetilde{C}$ is the normalization of $C$. 
\end{Lem}

\begin{proof}[Proof of Lemma \ref{taking a good curve}]
If $I_x=0$, there is nothing to prove. 
We assume $I_x\neq0$. 
First we reduce the general case to the case where $I_x\simeq\mathbb{F}_p$. 
Let $H$ be the subgroup of index $p$ of $I_x$ 
and $W'$ the quotient $W/H$. 
Then, the inertia subgroup $I_{W'/X,x}$ at $x$ of the Galois group $G/H$ of $W'\to X$ 
	is isomorphic to $\mathbb{F}_p$. 
If the $I_x\simeq\mathbb{F}_p$ case is known, 
we can find 
a regular scheme $\overline{C}$ of dimension one, 
a finite morphism $\overline{C}\to\overline{X}$, 
and a geometric point $u$ of $\overline{C}$ lying above $x$ 
such that 
$C=\overline{C}\times_{\overline{X}}X$ is not empty 
and 
the composite 
$\pi_1(\overline{C}_{(v)}\times_{\overline{C}}C,y')
\to\pi_1(\overline{X}_{(x)}\times_{\overline{X}}X,y)\to I_{W'/X,x}$ 
is surjective 
and the surjectivity of 
$\pi_1(\overline{C}_{(v)}\times_{\overline{C}}C,y')\to I_x$ 
follows. 

We assume $I_x\simeq\mathbb{F}_p$. 
By Zariski-Nagata's purity theorem, 
there exists a point $\xi\in\overline{X}$ of codimension one 
such that $x$ lies over the closure $\overline{\{\xi\}}$ and $W\to X$ is ramified at $\xi$. 
Let $V$ be the quotient $W/I_x$ 
and $\overline{V}$ the normalization of $\overline{X}$ in $V$. 
Take a geometric point $v$ of $\overline{V}$ lying over $x$ 
and a point $\eta\in\overline{V}$ lying over $\xi$ such that $v$ lies over the closure $\overline{\{\eta\}}$. 
Then, $W\to V$ is ramified at $\eta$ and $\eta$ is of codimension one. 
By Lemma \ref{KS}, 
We can find 
a regular scheme $\overline{C}$ of dimension one, 
a finite morphism $\overline{C}\to\overline{V}$, 
and a geometric point $u$ of $\overline{C}$ lying above $v$ 
such that 
$C=\overline{C}\times_{\overline{V}}V$ is not empty 
and the $\mathbb{F}_p$-torsor $W\times_VC\to C$ is ramified at $u$. 
Since an inertia group $I_{W/V,v}$ at $v$ is isomorphic to $\mathbb{F}_p$, 
we have the equality $I_{C\times_VW/C,u}=I_{W/V,v}$ of inertia groups. 
Thus $\pi_1(\overline{C}_{(u)}\times_{\overline{C}}C,y')\to I_{W/V,v}=I_x$ is surjective. 
\end{proof}

\subsection{General case}
We now prove Proposition \ref{p-power case} 
for the case where $X$ is general and $G\simeq\mathbb{Z}/p^e\mathbb{Z}$ 
for some prime number $p$ and some integer $e\geq0$: 
\begin{Lem}
\label{p-power cov}
Let $X$ be a regular scheme separated of finite type over $S$. 
Let $\mathcal{F}$ and $\mathcal{F}'$ be constructible complexes of 
$\Lambda$-modules and $\Lambda'$-modules respectively on $X$ 
and $W\to X$ a $\mathbb{Z}/p^e\mathbb{Z}$-torsor trivializing $\mathcal{F}$ and $\mathcal{F}'$. 
Assume that $\mathcal{F}$ and $\mathcal{F}'$ have universally the same conductors over $S$ 
and that $X$ has a regular compactification over $S$. 
Then $\mathcal{F}$ and $\mathcal{F}'$ have the same virtual wild ramification over $S$. 
\end{Lem}
\begin{proof}
Let $\overline{X}$ be a regular compactification of $X$ over $S$. 
It suffices to show that 
for every geometric point $x$ {\it over a closed point} of $\overline{X}$ 
and every pro-$p_x$ element $\sigma\in\pi_1(\overline{X}_{(\bar{x})}\times_{\overline{X}}X,y)$, 
we have $\Gamma^\sigma([\mathcal{F}_y])=\Gamma^\sigma([\mathcal{F}'_y])$. 
By Lemma \ref{taking a good curve}, 
we can find 
an $S$-curve $C$, an $S$-morphism $g:C\to X$, and a geometric point $v$ of $\overline{C}$ over $x$ 
such that the composite  
$\pi_1(\overline{C}_{(v)}\times_{\overline{C}}C,y')\to\pi_1(\overline{X}_{(x)}\times_{\overline{X}}X,y)\to I_x$ 
is surjective. 
By Lemma \ref{curve case}, 
$g^*\mathcal{F}$ and $g^*\mathcal{F}'$ have the same virtual wild ramification along $\overline{C}\setminus C$. 
Therefore the assertion follows. 
\end{proof}

We will need to consider the existence of a compactification to prove Proposition \ref{p-power case}. 
\begin{Lem}
\label{cptf}
Let $X$ be a normal $S$-scheme separated of finite type 
and $V\to X$ be a finite \'etale morphism. 
Then, for any normal compactification $\overline{V}$ of $V$ over $S$, 
there exists a normal compactification $\overline{X}$ of $X$ over $S$ 
such that the normalization of $\overline{X}$ in $V$ dominates $\overline{V}$. 
\end{Lem}
\begin{proof}
The proof is the same as the latter half of the proof of \cite[Proposition 2.1.1 (ii)]{V}. 
For the reader's convenience, we include the proof. 
Take an arbitrary compactification $Y$ of $X$ 
and a compactification $W$ of $V$ dominating $\overline{V}$ 
with a morphism $W\to Y$ extending $V\to X$. 
By \cite[I.5.2.2]{GR}, we can find 
a blow-up $Y'\to Y$ along a closed subscheme contained in $Y\setminus X$ 
such that $W'\to Y'$ is flat, 
where $W'\to W$ is the blow-up induced by $Y'\to Y$. 
Then, $W'\to Y'$ is finite, 
and hence it suffices to take $\overline{X}$ to be the normalization of $Y'$. 
\end{proof}
\begin{Lem}
\label{red to the cyclic case}
Let $X$ be a normal $S$-scheme separated of finite type, 
$G$ a finite group, 
$\mathcal{F}$ and $\mathcal{F}'$ smooth complexes of 
$\Lambda$-modules and $\Lambda'$-modules respectively on $X$, 
and $W\to X$ a $G$-torsor trivializing $\mathcal{F}$ and $\mathcal{F}'$. 
Consider subgroups $H_1,\ldots,H_k$ of $G$ such that 
any element $\sigma\in G$ of prime-power order belongs to 
$H_i$ for some $i=1,\ldots,k$. 
Set $V_i=W/H_i$ and 
assume $\mathcal{F}|_{V_i}$ and $\mathcal{F}'|_{V_i}$ 
have the same virtual wild ramification over $S$. 
Then $\mathcal{F}$ and $\mathcal{F}'$
have the same virtual wild ramification over $S$. 
\end{Lem}
\begin{proof}
For each $i$, we take a normal compactification $\overline{V_i}$ of $V_i$ over $S$ 
such that 
$\mathcal{F}|_{V_i}$ and $\mathcal{F}'|_{V_i}$ 
have the same virtual wild ramification along $\overline{V_i}\setminus V_i$. 
By Lemma \ref{cptf}, 
we can take a normal compactification $\overline{X}$ of $X$ over $S$ 
such that the normalization $\overline{V'_i}$ of $\overline{X}$ in $V_i$ dominates $\overline{V_i}$ 
for every $i$. 
We claim that $\mathcal{F}$ and $\mathcal{F}'$ have the same virtual wild ramification along $\overline{X}\setminus X$. 

Let $x$ be a geometric point of $\overline{X}$ and 
$\sigma$ be a pro-$p_x$ element of $\pi_1(\overline{X}_{(x)}\times_{\overline{X}}X,y)$, 
where $y$ is a geometric point. 
Fix a geometric point $w$ of the normalization of $\overline{X}$ in $W$ lying above $x$ 
and take $i$ such that the image $\bar{\sigma}$ of $\sigma$ in $G$ belongs to $H_i$. 
Then $\bar{\sigma}$ belongs to the inertia group $I_v\subset H_i$ at the image $v\to\overline{V'_i}$ of $w$ 
and hence comes from a pro-$p_v$ element $\sigma'$ of $\pi_1(\overline{V'_i}_{(v)}\times_{\overline{V'_i}}V_i,y')$. 
Since $\Gamma^\sigma([\mathcal{F}_y])=\Gamma^{\sigma'}([\mathcal{F}_{y'}])$ and the same for $\mathcal{F}'$, 
the assertion follows. 
\end{proof}


\begin{proof}[Proof of Proposition \ref{p-power case}]
Take all maximal cyclic subgroups $H_1,\ldots,H_k$ of prime-power order of $G$. 
By the assumption, we can apply Lemma \ref{p-power cov} to 
$\mathcal{F}|_{W/H_i}$ and $\mathcal{F}'|_{W/H_i}$ 
and thus $\mathcal{F}|_{W/H_i}$ and $\mathcal{F}'|_{W/H_i}$ 
have the same virtual wild ramification over $S$. 
Then, by Lemma \ref{red to the cyclic case}, 
$\mathcal{F}$ and $\mathcal{F}'$ have the same virtual wild ramification over $S$. 
\end{proof}

\begin{proof}[Proof of Theorem \ref{A to W}]
By virtue of Proposition \ref{W to A}, it suffices to show that 
having universally the same conductors implies having the same virtual wild ramification 
for $\mathcal{F}$ and $\mathcal{F}'$. 
By devissage, we may assume $X$ is regular and $\mathcal{F}$ and $\mathcal{F}'$ are smooth. 

Since every regular $S$-surface has a regular compactification over $S$ \cite{L}, 
the assertion follows from Proposition \ref{p-power case}. 
\end{proof}

\section{The Euler-Poincar\'e characteristic}
We prove Theorem \ref{A to X}. 
The following proposition is a consequence of the result of Deligne-Illusie. 
\begin{Prop}[cf. Proposition 0.2 in \cite{SY}]
\label{W to X}
Let 
$k$ be a perfect field 
and $\bar{k}$ an algebraic closure of $k$. 
Let $X$ be a scheme separated of finite type over $k$ 
and $\mathcal{F}$ and $\mathcal{F}'$ constructible complexes of 
$\Lambda$-modules and $\Lambda'$-modules respectively on $X$. 
Assume that $\mathcal{F}$ and $\mathcal{F}'$ have the same virtual wild ramification $(\text{over }k)$. 
Then, we have $\chi _c(X_{\bar{k}},\mathcal{F})=\chi _c(X_{\bar{k}},\mathcal{F}')$. 
\end{Prop}
\begin{Prop}
\label{EP and brtr}
Let $X$ be a normal connected scheme separated of finite type over an algebraically closed field, 
$\mathcal{F}$ be a smooth complex of $\Lambda$-module on $X$, 
and $W\to X$ a $G$-torsor, for a finite group $G$, trivializing $\mathcal{F}$. 
We denote the virtual representation of $G$ corresponding to $\mathcal{F}$ by $M$. 
\begin{enumerate}[{\rm (i)}]
\item {\rm \cite[Lemme 2.3 and Remarque 2.4]{I}}
We have 
$$
\chi_c(X,\mathcal{F})=\frac{1}{|G|}\sum_{\sigma\in G}\Tr(\sigma,H_c^*(W,\mathbb{Q}_l))\cdot\Trbr(\sigma,M),
$$
where $l$ is the characteristic of $\Lambda$ 
and for $\sigma\in G$ of order divided by $l$ we put $\Trbr(\sigma,-)=0$. 
\item {\rm \cite[3.3]{DL}}
$\Tr(\sigma,H_c^*(W,\mathbb{Q}_l))$ is an integer independent of $l\neq p$. 

\item
Let $\overline{X}$ be a normal compactification of $X$ 
and $\overline{W}$ a normalization of $\overline{X}$ in $W$. 
For an element $\sigma\in G$, if $\Tr(\sigma,H_c^*(W,\mathbb{Q}_l))\neq0$ 
then $\sigma$ is of order a power of $p$ and fixes some point of $\overline{W}$. 
\end{enumerate}
\end{Prop}
\begin{proof}
For (iii), see the beginning of the proof of \cite[Lemme 2.5]{I}. 
\end{proof}

\begin{proof}[Proof of Proposition \ref{W to X}]
The way of the proof is the same as that of \cite[Proposition 0.2]{SY}. 
We may assume $k$ is algebraically closed. 
By taking a decomposition $X=\coprod_i X_i$ as in Definition \ref{svwr} (ii), 
we may assume that $X$ is normal and connected and that $\mathcal{F}$ and $\mathcal{F}'$ are smooth complexes. 
Then by Proposition \ref{EP and brtr} (i)(iii), 
we have 
$$
\chi_c(X,\mathcal{F})=\frac{1}{|G|}\sum_{\sigma\in S}\Tr(\sigma,H_c^*(W,\mathbb{Q}_l))\cdot\Trbr(\sigma,M),
$$
where $S=\{\sigma\in G\mid\sigma\text{ is of }p\text{-power order and fixes some point of }\overline{W}\}$. 
Since $\chi_c(X,\mathcal{F})$ is an integer, by Lemma \ref{brtr}, 
$$
\chi_c(X,\mathcal{F})=\frac{1}{|G|}\sum_{\sigma\in S}\Tr(\sigma,H_c^*(W,\mathbb{Q}_l))\cdot
\frac{p\cdot\Gamma^\sigma(M)-\Gamma^{\sigma^p}(M)}{p-1}.
$$
Hence the assertion follows from Proposition \ref{EP and brtr} (ii). 
\end{proof}

We use the following lemma to reduce Theorem \ref{A to X} to the surface case. 
\begin{Lem}
\label{red to surface case}
Let $f:X\to Y$ be a morphism 
of schemes separated of finite type 
over an algebraically closed field 
and $\mathcal{F}$ and $\mathcal{F}'$ constructible complexes 
of $\Lambda$-modules and $\Lambda '$-modules respectively on $X$. 
Assume that 
for every smooth curve $C$ 
and every morphism $g:C\to Y$, 
we have $\chi_c(X\times_YC,g'^*\mathcal{F})=\chi_c(X\times_YC,g'^*\mathcal{F}')$, 
where $g'$ is the morphism $X\times_YC\to X$. 
Then, $Rf_!\mathcal{F}$ and $Rf_!\mathcal{F}'$ have universally the same conductors. 
\end{Lem}
\begin{proof}
By the assumption and the proper base change theorem, 
for every smooth curve $C$ over $k$ 
and every morphism $g:C\to Y$, 
we have $\chi_c(C,g^*Rf_!\mathcal{F})=\chi_c(C,g^*Rf_!\mathcal{F}')$. 
Hence we may assume that $f=\id$. 
Then the assertion follows from the proof of \cite[Lemma 3.3]{SY}. 
\end{proof}


\begin{proof}[Proof of Theorem \ref{A to X}]
We may assume $k$ is algebraically closed. 
We prove the assertion by induction on $\dim X$. 
First we assume $\dim X\leq2$. 
Then by Theorem \ref{A to W}, 
having universally the same conductors 
implies having the same virtual wild ramification. 
Hence the assertion follows Proposition \ref{W to X}. 

Assume $\dim X>2$. 
By the induction hypothesis, it suffices to show the assertion after replacing $X$ by a dense open subscheme. 
We can find a flat morphism $f:X\to Y$ to a surface $Y$ over $k$ by shrinking $X$ if necessary 
(for example, take an \'etale morphism $X\to \mathbb{A}^d$, by shrinking $X$, 
and compose with some projection $\mathbb{A}^d\to \mathbb{A}^2$). 
By the assumption, for each smooth curve $C$ and each morphism $g:C\to Y$, 
the pullbacks $g'^*\mathcal{F}$ and $g'^*\mathcal{F}'$ 
by the canonical morphism $g':X\times_YC\to X$ 
have universally the same conductors. 
Here $\dim X\times_YC<\dim X$ by the flatness of $f$ 
and we can apply the induction hypothesis to $g'^*\mathcal{F}$ and $g'^*\mathcal{F}'$. 
Hence we have 
$\chi_c(X\times_YC,g'^*\mathcal{F})=\chi_c(X\times_YC,g'^*\mathcal{F}')$. 
Therefore by Lemma \ref{red to surface case}, 
$Rf_!\mathcal{F}$ and $Rf_!\mathcal{F}$ have the universally the same conductors. 
By the $\dim X=2$ case, we have $\chi_c(X,\mathcal{F})=\chi_c(X,\mathcal{F}')$ 
and the assertion follows. 
\end{proof}

\begin{Def}[Definition 3.1 in \cite{SY}]
\label{def of X}
\upshape
Let $X$ be a separated scheme of finite type over a perfect field $k$. 
For constructible complexes $\mathcal{F}$ and $\mathcal{F}'$ 
of $\Lambda$-modules and $\Lambda'$-modules respectively on $X$, 
we say $\mathcal{F}$ and $\mathcal{F}'$ have {\it universally the same Euler-Poincar\'e characteristics} 
if for any scheme $Z$ separated of finite type over $k$ 
and any morphism $g:Z\to X$, 
we have 
$\chi_c(Z_{\bar{k}},g^*\mathcal{F})=\chi_c(Z_{\bar{k}},g^*\mathcal{F}')$.
\end{Def}

\begin{Cor}
\label{sc and sx}
Let $X$ be a separated scheme of finite type over a perfect field $k$ 
and $\mathcal{F}$ and $\mathcal{F}'$ constructible complexes 
of $\Lambda$-modules and $\Lambda'$-modules respectively on $X$. 
Consider the following three conditions. 

\begin{enumerate}[{\rm (i)}]
\item\label{have the svwr}
$\mathcal{F}$ and $\mathcal{F}'$ have the same virtual wild ramification. 

\item\label{have the sc}
$\mathcal{F}$ and $\mathcal{F}'$ have universally the same conductors. 

\item\label{have the sx}
$\mathcal{F}$ and $\mathcal{F}'$ have universally the same Euler-Poincar\'e characteristics.
\end{enumerate}

Then, the implications 
{\rm (\ref{have the svwr})}$\Rightarrow${\rm (\ref{have the sc})}$\Leftrightarrow${\rm (\ref{have the sx})} hold. 
Moreover, if $\dim X\leq2$, then these three conditions are equivalent. 
\end{Cor}

\begin{proof}
(\ref{have the svwr})$\Rightarrow$(\ref{have the sc}): 
This is nothing but Proposition \ref{W to A}. 

(\ref{have the sc})$\Rightarrow$(\ref{have the sx}): 
Since the property having universally the same conductors is preserved by pullbacks, 
this follows from Theorem \ref{A to X}. 

(\ref{have the sx})$\Rightarrow$(\ref{have the sc}): 
This is a special case of Lemma \ref{red to surface case} (Take $f=\id$). 

Finally, if $\dim X\leq2$, then the implication (\ref{have the sc})$\Rightarrow$(\ref{have the svwr}) 
is nothing but Theorem \ref{A to W}. 
\end{proof}

\begin{Cor}
\label{functoriality over a field}
Let $f:X\to Y$ be a morphism of schemes separated of finite type over $k$ 
and $\mathcal{F}$ and $\mathcal{F}'$ constructible complexes 
of $\Lambda$-modules and $\Lambda'$-modules respectively on $X$. 
Then, 
\begin{enumerate}[{\rm (i)}]
\item
{\rm \cite[Lemma 3.3.2]{SY}}
If $\mathcal{F}$ and $\mathcal{F}'$ have universally the same Euler-Poincar\'e characteristics 
then so do $Rf_!\mathcal{F}$ and $Rf_!\mathcal{F}'$. 
\item
If $\mathcal{F}$ and $\mathcal{F}'$ have universally the same conductors 
then so do $Rf_!\mathcal{F}$ and $Rf_!\mathcal{F}'$. 
\item
Assume $\dim Y\leq2$. 
Then, if $\mathcal{F}$ and $\mathcal{F}'$ have the same virtual wild ramification 
then so do $Rf_!\mathcal{F}$ and $Rf_!\mathcal{F}'$. 
\end{enumerate}
\end{Cor}
\begin{proof}
\begin{description}
\item{(ii)}
This follows from (i) and Corollary \ref{sc and sx} (ii)$\Leftrightarrow$(iii). 

\item{(iii)}
By (i)$\Rightarrow$(iii) of Corollary \ref{sc and sx}, 
$\mathcal{F}$ and $\mathcal{F}'$ have universally the same Euler-Poincar\'e characteristics. 
Then by (i), $Rf_!\mathcal{F}$ and $Rf_!\mathcal{F}'$ have universally the same Euler-Poincar\'e characteristics 
and the assertion follows from (iii)$\Rightarrow$(i) of Corollary \ref{sc and sx} (ii). 
\end{description}
\end{proof}

Here we mention that wild ramification of the restrictions to all curves also determines the characteristic cycle. 
In \cite[Definition 4.10]{S}, 
for a constructible complex $\mathcal{F}$ 
on a smooth scheme $X$ over a perfect field, 
the characteristic cycle $CC\mathcal{F}$ is defined as a cycle on the cotangent bundle $T^*X$. 

\begin{Cor}
Let $X$ be a smooth scheme over a perfect field 
and $\mathcal{F}$ and $\mathcal{F}'$ be constructible complexes 
of $\Lambda$-modules and $\Lambda '$-modules respectively on $X$. 
Assume that $\mathcal{F}$ and $\mathcal{F}'$ have universally the same conductors. 
Then, $CC\mathcal{F}=CC\mathcal{F}'$. 
\end{Cor}

\begin{proof}
By \cite[Proposition 3.4]{SY}, it suffices to prove that 
$\mathcal{F}$ and $\mathcal{F}'$ have universally the same Euler-Poincar\'e characteristics. 
Hence the assertion follows from Corollary \ref{sc and sx} (ii)$\Rightarrow$(iii). 
\end{proof}

\section{The Swan conductors of cohomology groups}
Let $S$ be an excellent trait with perfect residue field of characteristic $p>0$ 
and generic point $\eta=\Spec K$. 
For a scheme $X$ separated of finite type over $K$ 
and a constructible complex $\mathcal{F}$ of $\Lambda$-modules on $X$, we denote by $\Sw(X,\mathcal{F})$ 
the alternating sum $\sum (-1)^i\Sw H_c^i(X_{\bar{K}},\mathcal{F})$. 

The following proposition is a consequence of the result of Vidal (i.e, Proposition \ref{l indep fixed pt}): 
\begin{Prop}
\label{svwr to sgc}
Let $X$ be a separated scheme of finite type over $K$ 
and $\mathcal{F}$ and $\mathcal{F}'$ constructible complexes 
of $\Lambda$-modules and $\Lambda'$-modules respectively on $X$. 
If $\mathcal{F}$ and $\mathcal{F}'$ have the same virtual wild ramification over $S$ 
then we have $\Sw(X,\mathcal{F})=\Sw(X,\mathcal{F}')$. 
\end{Prop}
\begin{Prop}
\label{l indep fixed pt}
Assume $S$ is strictly local. 
Let $X$ be a normal separated scheme of finite type over $\eta=\Spec K$. 
We denote the structure morphism by $f$. 
Let $\mathcal{F}$ be a smooth complex of $\Lambda$-modules on $X$ 
and $W\to X$ a $G$-torsor, for a finite group $G$, trivializing $\mathcal{F}$. 
We denote by $M$ the virtual representation of $G$ corresponding to $\mathcal{F}$ 
and take a finite Galois extension $K'$ of $K$ with Galois group $I$ 
trivializing the smooth complex $Rf_!\mathcal{F}$. 
We put $H=G\times I$. 
Then $H$ acts on $H_c^*(W_{\bar{\eta}},\mathbb{Q}_l)$ in a natural way 
and the followings hold, where $\bar{\eta}$ is a geometric point over $\eta$ 
and $l$ the characteristic of the finite field $\Lambda$. 
\begin{enumerate}[{\rm (i)}]
\item {\rm \cite[Corollary 6.4]{V}}
For $\sigma\in I$, we have 
$$
\Tr(\sigma,H_c^*(X_{\bar{\eta}},\mathcal{F}))
=\frac{1}{|G|}\sum_{g\in G}
\Tr ((g,\sigma),H_c^*(W_{\bar{\eta}},\mathbb{Q}_l))\cdot
\Trbr(g,M).
$$

\item {\rm\cite[Proposition 4.2]{V}
}
For $h\in H$, 
the trace $\Tr (h,H_c^*(W_{\bar{\eta}},\mathbb{Q}_l))$ is independent of $l\neq p$. 

\item {\rm\cite[Corollary 6.5]{V}
}
Let $\overline{X}$ be a normal compactification of $X$ over $S$ 
and $\overline{W}$ the normalization of $W$ in $\overline{X}$. 
Then for $h=(g,\sigma)\in H$, 
if $\Tr (h,H_c^*(W_{\bar{\eta}},\mathbb{Q}_l))\neq0$, 
then $h$ is of order a power of $p$ 
and the induced map $g:\overline{W}\to\overline{W}$ fixes some closed point of the closed fiber of $\overline{W}$. 
\end{enumerate}
\end{Prop}

\begin{proof}[Proof of Proposition \ref{svwr to sgc}]
We have 
$$
\Sw(X,\mathcal{F})=\frac{1}{|I|}\sum_{\sigma\in I}
sw_I(\sigma)\cdot
\Trbr(\sigma,(Rf_!\mathcal{F})_{\bar{\eta}}). 
$$
By Proposition \ref{l indep fixed pt} (i), we have 
$$
\Sw(X,\mathcal{F})=\frac{1}{|H|}\sum_{h=(g,\sigma)\in H}
sw_I(\sigma)\cdot
\Tr (h,H_c^*(W_{\bar{\eta}},\mathbb{Q}_l))\cdot
\Trbr(g,M).
$$
Since $\Sw(X,\mathcal{F})$ is an integer, 
by Lemma \ref{brtr}, 
$$
\Sw(X,\mathcal{F})=\frac{1}{|H|}\sum_{\substack{h=(g,\sigma)\in H\\\text{ of $p$-power order}}}
sw_I(\sigma)\cdot
\Tr (h,H_c^*(W_{\bar{\eta}},\mathbb{Q}_l))\cdot
\frac{p\cdot\Gamma^g(M)-\Gamma^{g^p}(M)}{p-1}.
$$
Then, the assertion follows from Proposition \ref{l indep fixed pt} (ii) and (iii). 
\end{proof}

Now we prove the following lemma, which is needed 
to reduce Theorem \ref{sc to sgc} to the curve case (i.e. the $S$-surface case). 
\begin{Lem}
\label{red2}
Let $f:X\to Y$ be a morphism of separated schemes of finite type over $K$ 
and $\mathcal{F}$ and $\mathcal{F}'$ constructible complexes 
of $\Lambda$-modules and $\Lambda'$-modules respectively on $X$. 
Assume that 
$\mathcal{F}$ and $\mathcal{F}'$ have universally the same conductors over $S$ 
and that 
for any regular $S$-curve $\xi$ and any morphism $g:\xi\to Y$, 
we have $\Sw(X\times_Y\xi,\mathcal{F})=\Sw(X\times_Y\xi,\mathcal{F}')$. 
Then, $Rf_!\mathcal{F}$ and $Rf_!\mathcal{F}'$ have universally the same conductors over $S$. 
\end{Lem}


\begin{Lem}
\label{x over a local field}
Let $Z$ be a separated scheme of finite type over $K$ 
and $\mathcal{F}$ and $\mathcal{F}'$ constructible complexes 
of $\Lambda$-modules and $\Lambda'$-modules respectively on $Z$. 
If $\mathcal{F}$ and $\mathcal{F}'$ have universally the same conductors over $S$, 
then they have universally the same conductors over $\bar{\eta}$, 
where $\bar{\eta}$ is the spectrum of an algebraic closure of $K$. 
\end{Lem}
\begin{proof}
By Proposition \ref{W to A}, it suffices to show that 
for every smooth curve $C$ over $\bar{\eta}$ 
and every morphism $g:C\to Z\times_\eta\bar{\eta}$ over $\bar{\eta}$, 
$g^*p^*\mathcal{F}$ and $g^*p^*\mathcal{F}'$ have the same virtual wild ramification over $\bar{\eta}$, 
where $p$ is the projection $Z\times_\eta\bar{\eta}\to Z$. 

Take a smooth curve $C$ and a morphism $g:C\to Z\times_\eta\bar{\eta}$ over $\bar{\eta}$. 
We can find 
a finite extension $\eta'$ of $\eta$, 
a  smooth curve $C'$ over $\eta'$, 
and a morphism $g':C'\to Z\times_\eta\eta'$ over $\eta'$ 
such that $g'$ induces $g$ by base extension. 
Then by the assumption, 
$g'^*p'^*\mathcal{F}$ and $g'^*p'^*\mathcal{F}'$ have universally the same conductors over $S'$, 
where $p'$ is the projection $Z\times_\eta\eta'$. 
Since $C'$ is an $S$-surface, we can apply Theorem \ref{A to W}. 
Thus $g'^*p'^*\mathcal{F}$ and $g'^*p'^*\mathcal{F}'$ have the same virtual wild ramification over $S'$. 
Then the assertion follows from Lemma \ref{pullback}. 
\end{proof}

\begin{proof}[Proof of Lemma \ref{red2}]
We may assume $S$ is strictly local. 
By Lemma \ref{char of sc}, it suffices to show 
\begin{enumerate}[{\rm (i)}]
\item
For every geometric point $y$ of $Y$, we have $\rk Rf_!\mathcal{F}_y=\rk Rf_!\mathcal{F}'_y$. 
\item
For every finite extension $\xi$ of $\eta$ 
and every morphism $\xi\to Y$ over $\eta$, 
we have $\Sw Rf_!\mathcal{F}|_\xi=\Sw Rf_!\mathcal{F}'|_\xi$. 
\end{enumerate}

To show (i), we may assume $y$ is the spectrum of an algebraically closed field. 
Then the assertion follows from Lemma \ref{x over a local field} and Theorem \ref{A to X}. 

By Lemma \ref{conductor discriminant}, we have 
$\Sw(X\times_Y\xi,\mathcal{F})=\Sw(Rf_!\mathcal{F}|_{\xi})+\rk(Rf_!\mathcal{F}|_{\xi})\cdot(d+n-1)$ 
and the same equality for $\mathcal{F}'$. 
Hence the assertion follows from (i) and the assumption. 
\end{proof}

\begin{proof}[Proof of Theorem \ref{sc to sgc}]
We may assume $S$ is strictly local. 
We prove the assertion by induction on $\dim X$. 
If $\dim X\leq1$, by Theorem \ref{A to W}, 
$\mathcal{F}$ and $\mathcal{F}'$ have the same virtual wild ramification over $S$. Then the assertion follows from Proposition \ref{W to A}. 
Assume $\dim X\geq2$. 
By the induction hypothesis, it suffices to prove the assertion 
after replacing $X$ by a dense open subscheme. 
Then, we can take a flat morphism $f:X\to Y$ to a curve $Y$ over $K$ 
by shrinking $X$ if necessary. 
Then for every finite extension $\xi$ of $\eta$ 
and every morphism $\xi\to Y$, we have 
$\Sw(X\times_Y\xi,\mathcal{F})=\Sw(X\times_Y\xi,\mathcal{F}')$. 
In fact, since we have $\dim X\times_Y\xi<\dim X$, we can apply the induction hypothesis to 
$g^*\mathcal{F}$ and $g^*\mathcal{F}'$, where $g$ is the canonical morphism $X\times_Y\xi\to X$. 
Then, by Lemma \ref{red2}, $Rf_!\mathcal{F}$ and $Rf_!\mathcal{F}'$ have universally the same conductors over $S$. 
By the $\dim X=1$ case, we get $\Sw(Y,Rf_!\mathcal{F})=\Sw(Y,Rf_!\mathcal{F}')$ and the assertion follows. 
\end{proof}

\begin{Def}
\label{def of X}
\upshape
Let $X$ be a separated scheme of finite type over $K$. 
For constructible complexes $\mathcal{F}$ and $\mathcal{F}'$ 
of $\Lambda$-modules and $\Lambda'$-modules respectively on $X$, 
we say $\mathcal{F}$ and $\mathcal{F}'$ have {\it universally the same global conductors} over $S$ 
if for any scheme $Z$ separated of finite type over $K$ 
and any morphism $g:Z\to X$, 
we have 
$\chi_c(Z_{\bar{K}},g^*\mathcal{F})=\chi_c(Z_{\bar{K}},g^*\mathcal{F}')$ 
and $\Sw(Z,g^*\mathcal{F})=\Sw(Z,g^*\mathcal{F}')$. 
\end{Def}

\begin{Cor}
\label{sc and sgc}
Let $X$ be a separated scheme of finite type over $K$ 
and $\mathcal{F}$ and $\mathcal{F}'$ constructible complexes 
of $\Lambda$-modules and $\Lambda'$-modules respectively on $X$. 
Consider the following three conditions. 
\begin{enumerate}[{\rm (i)}]
\item\label{svw s}
$\mathcal{F}$ and $\mathcal{F}'$ have the same virtual wild ramification over $S$. 
\item\label{sc s}
$\mathcal{F}$ and $\mathcal{F}'$ have universally the same conductors over $S$. 
\item\label{sgc s}
$\mathcal{F}$ and $\mathcal{F}'$ have universally the same global conductors over $S$. 
\end{enumerate}

Then, the implications 
{\rm (\ref{svw s})}$\Rightarrow${\rm (\ref{sc s})}$\Leftrightarrow${\rm (\ref{sgc s})} hold. 
Moreover, if $\dim X\leq1$, then these three conditions are equivalent. 
\end{Cor}
\begin{proof}
(\ref{svw s})$\Rightarrow$(\ref{sc s}): 
This is nothing but Proposition \ref{W to A}. 

(\ref{sc s})$\Rightarrow$(\ref{sgc s}): 
Since the property having universally the same conductors is preserved by pullbacks, 
this follows from Theorem \ref{sc to sgc}. 

(\ref{sgc s})$\Rightarrow$(\ref{sc s}): 
We may assume $X$ is connected and finite over $\eta=\Spec K$ 
and it suffices to show $\rk\mathcal{F}=\rk\mathcal{F}'$ and $\Sw(\mathcal{F})=\Sw(\mathcal{F}')$. 
The first equality follows from the equality $\rk f_*\mathcal{F}=[X:\eta]\cdot\rk\mathcal{F}$. 
The second equality follows from Lemma \ref{conductor discriminant}. 

Finally when $\dim X\leq1$, then the implication (\ref{sc s})$\Rightarrow$(\ref{svw s}) 
is nothing but Theorem \ref{A to W}. 

%
%
%
\end{proof}


\begin{Cor}
\label{functoriality over a local field}
Let $f:X\to Y$ be a morphism between schemes separated of finite type over $K$ 
and $\mathcal{F}$ and $\mathcal{F}'$ constructible complexes 
of $\Lambda$-modules and $\Lambda'$-modules respectively on $X$. 
Then, 
\begin{enumerate}[{\rm (i)}]
\item
If $\mathcal{F}$ and $\mathcal{F}'$ have universally the same global conductors over $S$, 
then so do $Rf_!\mathcal{F}$ and $Rf_!\mathcal{F}'$. 
\item
If $\mathcal{F}$ and $\mathcal{F}'$ have universally the same conductors over $S$, 
then so do $Rf_!\mathcal{F}$ and $Rf_!\mathcal{F}'$. 
\item
Assume $\dim Y\leq1$. 
Then, if $\mathcal{F}$ and $\mathcal{F}'$ have the same virtual wild ramification over $S$, 
then so do $Rf_!\mathcal{F}$ and $Rf_!\mathcal{F}'$. 
\end{enumerate}
\end{Cor}
\begin{proof}
\begin{enumerate}[{\rm (i)}]
\item
For a scheme $Z$ separated of finite type over $K$ and a morphism $g:Z\to Y$, 
by the proper base change theorem, 
$\chi_c(Z_{\bar{K}},g^*Rf_!\mathcal{F})=\chi_c((Z\times_YX)_{\bar{K}},g'^*\mathcal{F})$ 
and $\Sw(Z,g^*Rf_!\mathcal{F})=\Sw(Z\times_YX,g'^*\mathcal{F})$, 
where $g'$ is the canonical morphism $Z\times_YX\to X$.  

\item
This follows from (i) and Corollary \ref{sc and sgc} (ii)$\Leftrightarrow$(iii). 

\item
By (i)$\Rightarrow$(iii) of Corollary \ref{sc and sgc}, 
$\mathcal{F}$ and $\mathcal{F}'$ have universally the same global conductors over $S$. 
Then by (i), $Rf_!\mathcal{F}$ and $Rf_!\mathcal{F}'$ have universally the same global conductors over $S$ 
and the assertion follows from (iii)$\Rightarrow$(i) of Corollary \ref{sc and sgc}.   
\end{enumerate}
\end{proof}

\end{document}